\documentclass[12pt]{article}
\usepackage[centertags]{amsmath}
\usepackage{amssymb,enumerate}
\usepackage{graphicx,amsfonts}
\usepackage{amsthm}
\input{amssym.def}
\newcommand{\be}{\begin{equation}}
\newcommand{\ee}{\end{equation}}
\newcommand{\bc}{\begin{center}}
\newcommand{\ec}{\end{center}}

\numberwithin{equation}{section}

\newtheorem{Definition}{Definition}[section]
\newtheorem{theorem}[Definition]{Theorem}
\newtheorem{definition}[Definition]{Definition}
\newtheorem{lemma}[Definition]{Lemma}

\newtheorem{corollary}[Definition]{Corollary}

\newcommand{\lc}{\mathcal{L}}
\newcommand{\rc}{\mathcal{R}}
\newcommand{\hc}{\mathcal{H}}
\newcommand{\jc}{\mathcal{J}}
\newcommand{\dc}{\mathcal{D}}
\newcommand{\lp}{\mathcal{L}^+}
\newcommand{\rp}{\mathcal{R}^+}
\newcommand{\hp}{\mathcal{H}^+}
\newcommand{\dpp}{\mathcal{D}^+}
\newcommand{\jp}{\mathcal{J}^+}
\newcommand{\ld}{\mathcal{L}^{\bullet}}
\newcommand{\rd}{\mathcal{R}^{\bullet}}
\newcommand{\hd}{\mathcal{H}^{\bullet}}
\newcommand{\dd}{\mathcal{D}^{\bullet}}
\newcommand{\jd}{\mathcal{J}^{\bullet}}

\newcommand{\lzd}{\mathbf{LZ}^{\bullet}}
\newcommand{\rzd}{\mathbf{RZ}^{\bullet}}

\newcommand{\vd}{\mathbf{D}}

\newcommand{\vn}{\mathbf{N}}
\newcommand{\vln}{\mathbf{LN}}
\newcommand{\vrn}{\mathbf{RN}}

\newcommand{\bi}{\mathbf{BI}}
\newcommand{\lbi}{\mathbf{LBI}}
\newcommand{\rbi}{\mathbf{RBI}}
\newcommand{\lqbi}{\mathbf{LQBI}}
\newcommand{\rqbi}{\mathbf{RQBI}}

\begin{document}
\title{\Large \bf On band orthorings}
\author{A. K. Bhuniya$^1$ and R. Debnath$^2$}
\date{}
\maketitle

\begin{center}
{\small\it $^1$ Department of Mathematics, Visva-Bharati, Santiniketan-731235, India.} \\
{\small\it $^2$ Department of Mathematics, Kurseong college, Kurseong-734203, India.}\\
{\small e-mail: anjankbhuniya@gmail.com$^1$; rajib.d6@gmail.com$^2$}
\end{center}

\begin{abstract}
A semiring $S$ which is a union of rings is called completely regular, if moreover, it is orthodox then $S$ is called an orthoring. Here we study the orthorings $S$ such that $E^+(S)$ is a band semiring. Every band semiring is a spined product of a left band semiring and a right band semiring with respect to a distributive lattice. A similar spined product decomposition for the band orthorings have been proved. The interval $[\mathbf{Ri}, \mathbf{BOR}]$ is lattice isomorphic to the lattice $\mathcal{L}(\mathbf{BI})$ of all varieties of band semirings, where $\mathbf{Ri}$ and $\mathbf{BOR}$ are the varieties of all rings and band orthorings, respectively.
\end{abstract}


\section{Introduction}
A semigroup is called completely regular if it is a union of groups, and Clifford if it is a semilattice of groups. Completely regular semigroups have been studied extensively. For an account of the theory of completely regular semigroups we refer to \cite{howie} and \cite{PR vol I}. Huge development of the theory of the completely regular semigroups has made the researchers interested to study the semirings which are union of rings.

Bandelt and Petrich \cite{banpet} studied the semirings which are subdirect products of a ring and a distributive lattice. Also they introduced a construction of strong distributive lattices of rings. In this construction, we need two families of semiring homomorphisms to define two binary operations addition and multiplication in semiring. Ghosh \cite{gho} refined this construction of strong distributive lattices of rings and he showed that only one family of homomorphisms is sufficient to define both the binary operations in a semiring. He continued to further characterizations of the semirings which are subdirect products of a ring and a distributive lattice. Also he studied the Clifford semirings as strong distributive lattices of rings.

In \cite{pastijn-guo}, Pastijn and Guo characterized the semirings which are unions of rings. Let us call such semirings completely regular. If $(S, +, \cdot)$ is a completely regular semiring, then every additive idempotent of $S$ is a multiplicative idempotent. As a consequence, it follows that if $S$ is an orthodox completely regular semiring then $E^+(S)$ [set of all additive idempotents of $S$] is an idempotent semiring. They proved that the the interval $[\mathbf{Ri}, \mathbf{O}]$ is isomorphic to the lattice $\mathcal{L}(\mathbf{I})$ of all varieties of idempotent semirings, where $\mathbf{Ri}$ and $\mathbf{O}$ are the varieties of all rings and orthodox completely regular semirings, respectively.

This lattice isomorphism as well as the inspiring success achieved in characterizing orthogroups by their bands of idempotents inspired us to characterize the orthodox completely regular semirings $S$ such that the idempotent semiring $E^+(S)$ is in a given variety of idempotent semirings. In this article we study the band orthorings, equivalently orthodox completely regular semirings $S$ such that $E^+(S)$ is a band semiring, an important class of idempotent semirings $I$ such that $\mathcal{D}^+$ is the least distributive lattice congruence on $I$ \cite{guo shum sen}.

Also we wanted to characterize the variety of all completely regular semirings as an equational class. In Section 2, we find a set of five equations that defines completely regular semirings.

In Section 3, we have characterized the idempotent semirings $I$ such that $\mathcal{L}^+$ [$\mathcal{R}^+$] is the least distributive lattice congruence on $I$ which we call the left [right] band semirings. Every band semiring is a spined product of a left band semiring and a right band semiring with respect to a distributive lattice.

In Section 4, a similar spined product decomposition for the band orthorings has been established. Every band orthoring is a spined product of a left band orthoring and a right band orthoring with respect to a distributive lattice of rings. Also we have characterized distributive lattice decompositions of the [left, right] band orthorings.

In Section 5, we establish a lattice isomorphism between the lattice $\mathcal{L}(\mathbf{BI})$ of all varieties of band semirings and the interval $[\mathbf{Ri}, \mathbf{BOR}]$, where $\mathbf{BOR}$ is the variety of all band orthorings.

\section{Preliminaries and basic results}
A semiring $(S,+,\cdot)$ is an algebra with two binary operations $+$ and $\cdot$ such that the reducts $(S,+)$ and $(S,\cdot)$ are semigroups and in which the two distributive laws
$$
x(y+z)=xy+xz, \quad (x+y)z=xz+yz
$$
are satisfied. If $e \in S$ is such that $e+e=e$, then $e$ is called an additive idempotent. We denote the set of all additive idempotents of $S$ by $E^+(S)$. For a semiring $(S, +, \cdot)$ the Green relations $\lc, \rc, \hc, \jc$ and $\dc$ on the additive [multiplicative] reduct $(S, +)[(S, \cdot)]$  will be denoted by $\lp, \rp, \hp, \jp$ and $\dpp [\ld, \rd, \hd, \jd, \dd]$. We refer to \cite{PR vol I} and \cite{howie} for the information concerning semigroup theory, \cite{Golan} for background on semirings and ~\cite{McKenzie} for notions concerning universal algebra and lattice theory.

A semiring $(S, +, \cdot)$ is called an idempotent semiring if both the additive reduct $(S, +)$ and the multiplicative reduct $(S, \cdot)$ are bands. Thus the class of all idempotent semirings can be regarded as a variety of semirings satisfying two additional identities:
\begin{align*}
x+x = x \; \textrm{and} \; x \cdot x = x.
\end{align*}
We denote the variety of all idempotent semirings by $\mathbf{I}$. There are many articles considering variety of idempotent semirings, for example, see \cite{PZ2000}, \cite{SGS}, \cite{Wang}, \cite{ZGS} and \cite{ZSG}.

On an idempotent semiring $(S, +, \cdot)$ one may introduce the relations $\leq^{l}_{+}, \, \leq^{l}_{\cdot}, \leq^{r}_{+}, \, \leq^{r}_{\cdot}$ and $\leq_{+}, \, \leq_{\cdot}$ by the following: for $a, b \in S$,
\begin{align*}
& a \mathrel\leq^{l}_{+} b \hspace{.5cm}  \Leftrightarrow \hspace{.5cm} b=a+b; \hspace{1cm}
a \mathrel\leq^{l}_{\cdot} b \hspace{.5cm} \Leftrightarrow \hspace{.5cm} a=ba; \\
& a \mathrel\leq^{r}_{+} b \hspace{.5cm}  \Leftrightarrow \hspace{.5cm} b=b+a; \hspace{1cm}
a \mathrel\leq^{r}_{\cdot} b \hspace{.5cm} \Leftrightarrow \hspace{.5cm} a=ab; \\
& \hspace{1.4cm} \leq_{+} = \leq^{l}_{+} \cap \leq^{r}_{+} \hspace{.5cm} \textrm{and} \hspace{.5cm}
\leq_{\cdot} = \leq^{l}_{\cdot} \cap \leq^{r}_{\cdot}.
\end{align*}
The relations $\leq^{l}_{+}, \, \leq^{l}_{\cdot}, \leq^{r}_{+}, \, \leq^{r}_{\cdot}$ are quasi-orders and the relations $\leq_{+}$ and $\leq_{\cdot}$ are partial orders~\cite{nambooripad}.

If $(S, +, \cdot)$ is an idempotent semiring, then $\dpp$ and $\dd$ are given by: for $a, b \in S$,
\begin{align*}
& a \mathrel\dpp b \, \Leftrightarrow \, a+b+a=a, \ b+a+b=b; \\
& a \mathrel\dd b \, \Leftrightarrow \, aba=a, \ bab=b
\end{align*}
while $\rp$ and $\rd$ are defined by
\begin{align*}
& a \mathrel\rp b \, \Leftrightarrow \, a+b=b, \ b+a=a; \\
& a \mathrel\rd b \, \Leftrightarrow \, ab=b, \ ba=a
\end{align*}
and $\lp, \ld$ are defined dually.

Let $S$ be a semiring which is a union of rings. Then $(S, +)$ is a completely regular semigroup and hence can be regarded as a $(2, 1)$ algebra $(S, +, ')$ which satisfies:
\begin{align}
x = x+x'+x; \\
x+x' = x'+x; \\
(x')' = x.
\end{align}

Also by Theorem 1.3 \cite{pastijn-guo}, $\hp$ is the least idempotent semiring congruence on $S$. Thus $(S, +, ')$ is in fact a band of abelian groups. We refer to \cite{pastijn-trotter} for the properties of band of abelian groups. Then for all $x, y \in S$, $x+y^o \mathrel\hp x^o+y$, where $x^o=x+x'$ implies that
\begin{equation}
x+y^o+x^o+y=x^o+y+x+y^o.
\end{equation}

If $S=\cup R_{\alpha}$ and $x \in R_{\alpha}$ then $x^o$ is the zero element of the ring $R_{\alpha}$ \cite{SMScr}; and hence
\begin{equation}
xx^o=x^o.
\end{equation}

Now we show that the equations $(2.1) - (2.5)$ forms an equational basis for the variety of all semirings which are union of rings.
\begin{theorem}                                                  \label{hpcong}
Let $S$ be a semiring. Then the following conditions are equivalent:
\\ (i) \ $S$ is a union of rings;
\\ (ii) \ $S$ satisfies equations $(2.1) - (2.5)$;
\\ (iii) \ $\hp$ is an idempotent semiring congruence on $S$ and each $\hp$-class is a ring;
\\ (iv) \ $S$ is an idempotent semiring of rings.
\end{theorem}
\begin{proof}
Equivalence of (i), (iii) and (iv) follows from Theorem 1.3 \cite{pastijn-guo}. $\bf {(i) \Rightarrow (ii)}$ follows from the above discussions.
\\{\bf (ii) $\Rightarrow$ (i):}  Let $H$ be an $\hp$-class in $S$. Since $S$ satisfies equations (2.1), (2.2), (2.3) and (2.5), it follows by Theorem 3.6 \cite{SMScr}, that $H$ is a subsemiring of $S$ such that $(H, +)$ is a group. Now for $x, y \in H$, $x^o=y^o$ implies by $(2.4)$ that $x+y=x+x^o+y^o+y=x+y^o+x^o+y=x^o+y+x+y^o=y+x$; and hence every $\hp$-class is a ring.
\end{proof}

\begin{definition}
\begin{enumerate}
\item[(i)]
A semiring $S$ is said to be completely regular if it satisfies either of the four equivalent conditions in the above theorem.
\item[(ii)]
A completely regular semiring $S$ is called an orthoring if $E^+(S)$ is a subsemiring of $S$.
\end{enumerate}
\end{definition}

Readers should note that Sen, Maity and Shum \cite{SMScr} used the phrase `completely regular semiring' to mean union of skew-rings (rings without additive commutative). Also there are many other articles \cite{guo shum sen}, \cite{SMSclifford} etc. following this convention.

Since every additive idempotent of a completely regular semiring $S$ is also a multiplicative idempotent \cite{pastijn-guo}, if $S$ is an orthoring then the subsemiring $E^+(S)$ is an idempotent semiring. We call an idempotent semiring $S$ a \emph{rectangular} [\emph{left zero, right zero}] idempotent semiring if the additive reduct $(S, +)$ is a rectangular [left zero, right zero] band. By a \emph{b-lattice} we mean an idempotent semiring with an additive commutative reduct. The Green's relation $\dpp$ is a congruence on every idempotent semiring $S$, and hence every idempotent semiring is a b-lattice of rectangular idempotent semirings.

We call a semiring a \emph{rectangular ring} if it is a direct product of a rectangular idempotent semiring and a ring. Similarly we define a \emph{left ring} and a \emph{right ring}. We left it to the readers to check that a semiring $S$ is a rectangular [left, right] ring if and only if it is completely regular and the additive reduct $(S, +)$ of $S$ is a rectangular [left, right] group.

We omit the proof of the following result, since it is similar to that of Theorem 4.5 \cite{debnath-maity-bhuniya}.
\begin{theorem}
Let $S$ be a completely regular semiring. Then $S$ is an orthoring if and only if it is a b-lattice of rectangular rings.
\end{theorem}

Clearly this result is stemmed out from the semilattice decomposition of every band into rectangular bands. Since the variety $\mathbf{D}$ of all distributive lattices is a subvariety of the variety $\mathbf{Sl}^+$ of all b-lattices, it is natural to ask for the characterization of the semirings which are distributive lattices of rectangular rings.

Our experience in semigroups tells us to characterize first the idempotent semirings which are distributive lattices of rectangular idempotent semirings. In the following section we characterize such idempotent semirings.

\section{Band semirings}
Let $S$ be an idempotent semiring. Then the additive reduct $(S, +)$ is a band and hence $\mathcal{D}^+$ is the least semilattice congruence on $(S, +)$. If $\eta$ is the least distributive lattice congruence on $S$, then $(S/\eta, +)$ is a semilattice and hence $\mathcal{D}^+ \subseteq \eta$.

Sen, Guo and Shum \cite{guo shum sen} proved the following equivalent conditions:
\begin{lemma}
Let $S$ be an idempotent semiring. Then the following conditions are equivalent:
\begin{itemize}
\item[(i)]
$\dpp$ is the least distributive lattice congruence on $S$;
\item[(ii)]
$S$ is a distributive lattice of rectangular idempotent semirings;
\item[(iii)]
$S$ satisfies the following two identities:
\begin{center}
$x+xy+x = x$ and $x+yx+x = x$.
\end{center}
\end{itemize}
\end{lemma}

\begin{definition}
An idempotent semiring is said to be a band semiring if it satisfies either of the equivalent conditions of the above lemma.
\end{definition}

\newpage
In this article, we consider the following varieties of idempotent semirings.

\begin{center}
\begin{table}

\caption{Certain subvarieties of idempotent semirings}
\begin{tabular}{ll}
  Notation & Defining identity within $\mathbf{I}$ \\
  \hline
  $\mathbf{Sl}^+$ & $xy+yx = yx+xy$,\\
  $\mathbf{R}^+$ & $x+y+x = x$,\\
  $\mathbf{R}^{\bullet}$ & $xyx = x$,\\
  $\mathbf{LZ}^+$ & $x+y = y$,\\
  $\mathbf{RZ}^+$ & $x+y = x$,\\
  $\mathbf{LZ}^{\bullet}$ & $xy = y$,\\
  $\mathbf{RZ}^{\bullet}$ & $xy = x$,\\
  $\mathbf{BI}$ & $x+xy+x = x , x+yx+x = x$,\\
  $\mathbf{LQBI}$ & $x+xy+x = x$,\\
  $\mathbf{RQBI}$ & $x+yx+x = x$,\\
  $\mathbf{N}$ & $x+xyx+x = x$,\\
  $\mathbf{LN}$ & $x+xyx = x$,\\
  $\mathbf{RN}$ & $xyx+x = x.$
  \end{tabular}
\end{table}
\end{center}

We call an idempotent semiring in $\mathbf{LQBI}$ and $\mathbf{RQBI}$ a \emph{left quasi band semiring} and a \emph{right quasi band semirings}, respectively. The idempotent semirings in $\mathbf{Sl}^+$ is referred as \emph{b-lattice}. If the additive reduct $(B, +)$ of a band semirings $B$ is commutative, then the multiplicative reduct $(B, \cdot)$ is also commutative \cite{guo shum sen}. Thus $\mathbf{BI}\cap \mathbf{Sl}^+=\mathbf{D}$, the variety of distributive lattices.

\begin{lemma}
An idempotent semiring $S \in \vn$ if and only if it satisfies the identity:
\begin{equation}
xz+xyz+xz = xz.
\end{equation}
\end{lemma}
\begin{proof}
First assume that $S \in \vn$. Then $xz = (x+xyzx+x)z = xz+xyzxz+xz = xz+xyzx(z+zxyz+z)+xz = xz+xyzxz+xyz+xyzxz+xz$ implies that $xz+xyz+xyzxz+xz = xz$ [add $xyz+xyzxz+xz$ to both sides from right]. Similarly this implies that $xz = xz+xyz+xz$.

Converse follows directly.
\end{proof}

Let $S$ be an idempotent semiring. Define a binary relation $\sigma$ on $S$ by: for $a, b \in S$,
\begin{center}
$a \mathrel\sigma b$ if and only if $aba=aba+a+aba$ and $bab=bab+b+bab$.
\end{center}

Sen, Bhuniya and Debnath \cite{sen-bhuniya-debnath} proved that $\sigma$ is the least distributive lattice congruence on every idempotent semiring $S\in \vn$ and the transitive closure of $\sigma$, denote by $\eta$ and given by:
\begin{align*}
a \mathrel\eta b  & \ \ \textrm{if and only if there exists} \ x \in S \ \textrm{such that} \\
            & \hspace{2cm} axbxa=axbxa+a+axbxa \ \textrm{and} \ bxaxb=bxaxb+b+bxaxb
\end{align*}
is the least distributive distributive congruence on every idempotent semiring $S$.

The following equivalent characterization of band semirings is useful.
\begin{lemma}                                                             \label{bi=x+yxy+x}
Let $S$ be an idempotent semiring. Then the following conditions are equivalent.
\begin{enumerate}
\item \vspace{-.3cm}
$S \in \bi$;
\item \vspace{-.4cm}
$S$ satisfies the identity:
\begin{equation}                                              \label{identity for dpp}
x+yxy+x = x.
\end{equation}
\end{enumerate}
\end{lemma}
\begin{proof}
$\mathbf{(1) \Rightarrow (2):}$ Let $a, b \in S$. Then $a+ab+a=a$ and $a+ba+a=a$. This implies that $ba+bab+ba=ba$ and we get
\begin{align*}
& a+ba+a=a+ba+bab+ba+a \\
\Rightarrow \ & a=a+ba+bab+ba+a \\
\Rightarrow \ & a+bab+ba+a=a+ba+bab+ba+a \\
\Rightarrow \ & a+bab+ba+a=a \\
\Rightarrow \ & a+bab+ba+a=a+bab+a \\
\Rightarrow \ & a=a+bab+a.
\end{align*}
Thus (1) implies (2).
\\ $\mathbf{(2) \Rightarrow (1):}$ \ Assume that (2) holds. Let $a, b \in S$. Then $a+bab+a=a$. This implies that $a+abab+a=a$ and so $a+ab+a=a$. Similarly $a+ba+a=a$. Hence $S \in \bi$.
\end{proof}

Following result was proved by Sen, Guo and Shum. The characterization of the least distributive lattice congruence on an idempotent semiring as we found in \cite{sen-bhuniya-debnath}, enables us to prove the same in a generalized way [in the sense that the same technique can be used to characterize the idempotent semirings $S$ such that $\lp$ and $\rp$ are the least distributive lattice congruence on $S$]. So we would like to include here the following new proof.
\begin{theorem}~\cite{SGS} \label{dpp least}
Let $S$ be an idempotent semiring. Then the following conditions are equivalent.
\begin{enumerate}
\item \vspace{-.3cm}
$\dpp$ is the least distributive lattice congruence on $S$;
\item \vspace{-.4cm}
$S \in \bi$.
\end{enumerate}
\end{theorem}
\begin{proof}
$\mathbf{(1) \Rightarrow (2):}$ \ Let $a, b \in S$. Then $(a+ab) \mathrel\dpp a$ implies that $a=a+(a+ab)+a$ and so $a=a+ab+a$. Similarly $a=a+ba+a$. Thus $S \in \bi$.
\\ $\mathbf{(2) \Rightarrow (1):}$ \ Let $a, b \in S$ such that $a \mathrel\dpp b$. Then $a=a+b+a$ implies that $bab=bab+b+bab$. Similarly $b=b+a+b$ implies that $aba=aba+a+aba$. Hence by Lemma 2.1, $a \mathrel\sigma b$ and so $a \mathrel\eta b$. Thus $\dpp \subseteq \eta$.

Now let $a \mathrel\sigma b$. Then $aba=aba+a+aba$ and $bab=bab+b+bab$. Since $S \in \bi$, so $a=a+bab+a$ and $b=b+aba+b$. Then $aba=aba+a+aba$ implies that
\begin{align*}
& b+aba+b=b+aba+a+aba+b \\
\Rightarrow \ & b=b+aba+a+aba+b \\
\Rightarrow \ & b+a+aba+b=b+aba+a+aba+b \\
\Rightarrow \ & b+a+aba+b=b \\
\Rightarrow \ & b+a+b=b.
\end{align*}
Similarly $bab=bab+b+bab$ implies that $a=a+b+a$. Thus $a \mathrel\dpp b$. Hence $\sigma \subseteq \dpp$. Since $\dpp$ is a congruence on $S$, we have that $\eta \subseteq \dpp$. Therefore $\eta = \dpp$.
\end{proof}

Now we characterize the idempotent semirings $S$ such that $\lp$ is the least distributive lattice congruence on $S$. First we prove the following lemma.
\begin{lemma}  \label{l-plus lemma}
Let $S$ be an idempotent semiring. Then the following conditions are equivalent.
\begin{enumerate}
\item \vspace{-.3cm}
$S \in \vln$;
\item \vspace{-.4cm}
$S \in \vn$ and $\rp \subseteq \dd$.
\end{enumerate}
\end{lemma}
\begin{proof}
$\mathbf{(1) \, \Rightarrow \, (2):}$ \, Let $a, b \in S$ such that $a \mathrel\rp b$. Then $a=b+a$ and $b=a+b$. Now
\begin{align*}
& a=a(a+b)a=aba \\
\textrm{and} \hspace{1cm} & b=b(b+a)b=bab
\end{align*}
implies that $a \mathrel\dd b$. Thus $\rp \subseteq \dd$. Also it follows trivially that $S \in \vn$.
\\$\mathbf{(2) \, \Rightarrow \, (1):}$ \, Let $a, b \in S$. Then $(a+b) \mathrel\rp (a+b+a)$. Then $\rp \subseteq \dpp$ implies that $(a+b) \mathrel\dpp (a+b+a)$. This implies that
\begin{align*}
& (a+b+a)(a+b)(a+b+a)=(a+b+a) \\
\Rightarrow \, & a(a+b+a)(a+b)(a+b+a)a=a(a+b+a)a \\
\Rightarrow \, & a(a+b+a)(a+b)(a+b+a)a=a, \hspace{1cm} (\textrm{since} \, S \in \vn) \\
\Rightarrow \, & a(a+b+a)(a+b)(a+b+a)a=a(a+b+a)(a+b)a \\
\Rightarrow \, & a=a(a+b+a)(a+b)a \\
\Rightarrow \, & a(a+b)a=a(a+b+a)(a+b)a \\
\Rightarrow \, & a(a+b)a=a.
\end{align*}
Thus $S\in \vln$.
\end{proof}
\begin{theorem}                                        \label{lp least lat cong}
Let $S$ be an idempotent semiring. Then the following conditions are equivalent.
\begin{enumerate}
\item
$\lp$ is the least distributive lattice congruence on $S$;
\item
$S \in \vln$ and $\dd \subseteq \lp$;
\item
$S \in \vn$ and $\rp \subseteq \dd \subseteq \lp$;
\item
$\leq^{l}_{+} \subseteq \leq_{\cdot}$ for $S$;
\item
$S$ satisfies the identity
\begin{equation}
x = (x+y)x(x+y);
\end{equation}
\item
$S$ satisfies the identity
\begin{equation}
 x = x+yxy.                                                  \label{identity for lp}
\end{equation}
\end{enumerate}
\end{theorem}
\begin{proof}
Equivalence of (2) and (3) follows from the Lemma \ref{l-plus lemma}. Equivalence of (5) and (6) follows trivially.
\\ $\mathbf{(1) \, \Rightarrow \, (2):}$ \, Assume that (1) holds. Let $a, b \in S$. Then the absorption law in $S/\lp$ yields that
$$ a(a+b) \mathrel\lp a.$$
Therefore $ a=a+a(a+b)=a+ab $ and so $a=a+aba$ for all $a, b \in S$. Thus $S \in \vln$. Since $\dd$ is the least semilattice congruence on the additive redact $(S, \cdot)$ so $\dd \subseteq \lp$. Thus (2) holds.
\\ $\mathbf{(2) \, \Rightarrow \, (5):}$ \, Let $a, b \in S$. Then $ab \mathrel\dd ba$ and so $ab \mathrel\lp ba$. Therefore
\begin{align*}
& ab=ab+ba \\
\textrm{and}\, \, & ba=ba+ab.
\end{align*}
Since $S \in \vln$, so $a=a(a+b)a$. This implies that
\begin{align*}
& a=(a+ab)(a+ba) \\
\Rightarrow \, & a=a+(a+ab)(a+ba)\\
\Rightarrow \, & a=a+(a+ab+ba)(a+ba+ab) \\
\Rightarrow \, & a=a+aba+ab+ba+bab \\
\Rightarrow \, & a=a+ab+ba+bab \\
\Rightarrow \, & a=(a+b)a(a+b)
\end{align*}
for all $a, b \in S$. Thus (5) follows.
\\ $\mathbf{(5) \, \Rightarrow \, (1):}$ \, For every $a, b \in S$, we have $a=(a+b)a(a+b)=a+bab$. This implies that $a=a+bab+a$. Therefore $\dpp$ is the least distributive lattice congruence on $S$, by Theorem \ref{dpp least}. Again for every $a, b \in S$, \,
\begin{align*}
& a=(a+b)a(a+b)  \\
\Rightarrow \, & a=a(a+b)a(a+b)a  \\
\Rightarrow \, & a=a(a+b)a.
\end{align*}
Now let $a, b \in S$ such that $a \mathrel\dpp b$. Then $a=a+b+a$ and $b=b+a+b$. Now
\begin{align*}
& a=a+b+a \\
\Rightarrow \, & a(a+b)a=(a+b+a)(a+b)(a+b+a) \\
\Rightarrow \, & a=a+b.
\end{align*}
Similarly $b=b+a$. Hence $a \mathrel\lp b$ and so $\dpp \subseteq \lp$. Therefore $\dpp=\lp$. Thus $\lp$ is the least distributive lattice congruence on $S$.
\\ $\mathbf{(4) \, \Rightarrow \, (5):}$ \, Let $a, b \in S$. Then $a \mathrel\leq^{l}_{+} a+b$. This implies that $a \mathrel\leq_{\cdot} a+b$. Then
$$ (a+b)a=a=a(a+b). $$
and so $a=(a+b)a(a+b)$. Thus $S$ satisfies the identity:
$$ x = (x+y)x(x+y).$$
\\ $\mathbf{(5) \, \Rightarrow \, (4):}$ \, Let $a, b \in S$ such that $a \mathrel\leq^{l}_{+} b$. Then $b=a+b$. Also $a, b \in S$ implies that $a=(a+b)a(a+b)$. This implies that
\begin{align*}
& a=bab \\
\Rightarrow \, & ba=a=ab
\end{align*}
and so $a \mathrel\leq_{\cdot} b$. Therefore $\leq^{l}_{+} \subseteq \leq_{\cdot} $.
\end{proof}

\begin{corollary}    \label{rp least lat cong}
Let $S$ be an idempotent semiring. Then the following conditions are equivalent.
\begin{enumerate}
\item
$\rp$ is the least distributive lattice congruence on $S$;
\item
$S \in \vrn$ and $\dd \subseteq \rp$;
\item
$S \in \vn$ and $\lp \subseteq \dd \subseteq \rp$;
\item
$\leq^{r}_{+} \subseteq \leq_{\cdot}$ for $S$;
\item
$S$ satisfies the identity
\begin{equation}
x = (y+x)x(y+x);
\end{equation}
\item
$S$ satisfies the identity
\begin{equation}
 x = yxy+x.                                 \label{identity for rp}
\end{equation}
\end{enumerate}
\end{corollary}

\begin{definition}
An idempotent semiring $S$ is called a left [right] band semiring if $\lp$ [$\rp$] is the least distributive lattice congruence on $S$.
\end{definition}

Thus, by Theorem \ref{lp least lat cong} and Corollary \ref{rp least lat cong}, an idempotent semiring $S$ is a left [right] band semiring if it satisfies the additional identity:
\begin{equation}
 x=x+yxy \; [x = yxy+x].
\end{equation}

We denote the variety of all left [right] band semirings by $\lbi$ [$\rbi$].
\begin{theorem}
$\lbi = \mathbf{LZ}^{+} \circ \vd$.
\end{theorem}
\begin{proof}
First assume that $S \in \mathbf{LZ}^{+} \circ \mathbf{D}$. Then there exists a congruence $\delta$ on $S$ such that $S/\delta \in \vd$ and each $\delta$-class belongs to $\mathbf{LZ}^{+}$. Let $a, b \in S$. Since $S/\delta \in \vd$, we have $a(a+b) \mathrel\delta a$ and $(a+b)a \mathrel\delta a$. Since the $\delta$-class of $a$ is in $\mathbf{LZ}^+$, we have
$$ a=a+(a+b)a=(a+b)a \hspace{.5cm} \textrm{and} \hspace{.5cm} a=a+a(a+b)=a(a+b). $$
Then $a=(a+b)aa(a+b)=(a+b)a(a+b)$ and it follows, by the equation (\ref{identity for lp}), that $S \in \lbi$. Thus $\mathbf{LZ}^{+} \circ \vd \subseteq \lbi$.

The reverse inclusion follows trivially. Hence $\lbi = \mathbf{LZ}^{+} \circ \vd$.
\end{proof}

The left-right dual of this theorem reads as:
\begin{theorem}
$\rbi = \mathbf{RZ}^{+} \circ \vd$.
\end{theorem}

Thus for every $S \in \mathbf{LBI}$ $[\mathbf{RBI}]$ the additive reduct $(S, +)$ is a semilattice of left [right] zero band and so a left [right] regular band. In \cite{Wang}, Wang, Zhou and Guo proved that the additive reduct $(S, +)$ of a band semiring is a regular band. Hence, by Theorem 3.6 \cite{SGS}, we have the following result.
\begin{theorem}                                  \label{dp=lp join rp}
An idempotent semiring $S$ is a band semiring if and only if it is a spined product of a left band semiring and a right band semiring with respect to a distributive lattice.
\end{theorem}

Thus $\bi \subseteq \lbi \vee \rbi$. Also $\lbi \subseteq \bi$ and $\rbi \subseteq \bi$, by Lemma \ref{bi=x+yxy+x}, and hence $\lbi \vee \rbi \subseteq \bi$.
\begin{corollary}
$\bi=\lbi \vee \rbi$.
\end{corollary}

Now we give some properties of left and right quasi band semirings.
\begin{lemma}
For an idempotent semiring $S$ the following conditions are equivalent:
\begin{enumerate}
\item \vspace{-.3cm}
$S \in \lqbi$;
\item \vspace{-.4cm}
$S/\dpp \in \lzd \circ \vd$.
\end{enumerate}
\end{lemma}
\begin{proof}
$\mathbf{(1) \, \Rightarrow \, (2):}$ \, $\dpp$ is a congruence on every idempotent semiring. Let $S \in \lqbi$. Then $\dpp$ is a congruence on $S$ and $\dpp_{S/\dpp}= \Delta$, the equality relation. Thus $S/\dpp \in \lqbi$ and $\dpp_{S/\dpp} \subseteq \ld_{S/\dpp}$. Therefore $S/\dpp \in \lzd \circ \vd$.
\\ $\mathbf{(2) \, \Rightarrow \, (1):}$ \, Assume that $S/\dpp \in \lzd \circ \vd$. Then $\ld_{S/\dpp}$ is the least distributive lattice congruence on $S/\dpp$. Let $\overline{a}$ be the $\dpp$-class of $a$ in $S$. Then for all $a, b \in S$, we have that
\begin{align*}
(\overline{a}+\overline{a}\overline{b}) \mathrel\ld \overline{a} \, & \Rightarrow \, \overline{a}=\overline{a+ab} \\
                                        & \Rightarrow \, a=a+ab+a
\end{align*}
and so $S \in \lqbi$.
\end{proof}
\begin{corollary}  \label{corollary}
$\lzd \circ \vd \subseteq \lqbi$
\end{corollary}
\begin{theorem}
$\lqbi = \mathbf{R}^{+} \circ (\lzd \circ \vd)$.
\end{theorem}
\begin{proof}
Let $S \in \lqbi$. Then by the above lemma, $S/\dpp \in \lzd \circ \vd$. Hence $\lqbi \subseteq \mathbf{R}^{+} \circ (\lzd \circ \vd)$. If $S \in \mathbf{R}^{+} \circ (\lzd \circ \vd)$ then there exists a congruence $\rho$ on $S$ such that $S/\rho \in \lzd \circ \vd$ and each $\rho$-class is in $\mathbf{R}^{+}$. By Corollary \ref{corollary}, $S/\rho \in \lqbi$. Hence for all $a, b \in S$,
$$ (a+ab+a) \mathrel\rho a. $$
Since each $\rho$-class is in $\mathbf{R}^{+}$, this implies that $a=a+ab+a$. Thus $\lqbi = \mathbf{R}^{+} \circ (\lzd \circ \vd)$.
\end{proof}

The left-right dual of this theorem reads as:
\begin{theorem}
$\rqbi = \mathbf{R}^{+} \circ (\rzd \circ \vd)$.
\end{theorem}

\section{Band orthorings}
\begin{definition}
A completely regular semiring $S$ said to be a (left, right) band orthoring if $E^+(S)$ is a (left, right) band semiring.
\end{definition}

In the following theorem we show that the distributive lattice decomposition of a band orthoring $S$ is governed by that of the band semiring $E^+(S)$.
\begin{theorem}
A completely regular semiring $S$ is a band orthoring if and only if it is a distributive lattice of rectangular rings.
\end{theorem}
\begin{proof}
First suppose that $S$ be a band orthoring. Then $\mathcal{J}^{+}$ is a b-lattice congruence on $S$ \cite{SMScr}.
Now for $a,b\in S$, we have $ab = (a^{o} + a)b = a^{o}b + ab = a^{o}(b+b^{o})+ ab = a^{o}b +a^{o}b^{o} + ab = a^{o}b +a^{o}b^{o} +b^{o}a^{o} +a^{o}b^{o} + ab = a^{o}b +a^{o}b^{o} + (b+b^{'})(a+a^{'})+ a^{o}b^{o}+ ab = x + ba + y$ for some $x, y \in S$.  Similarly, $ba = u + ab + v $ for some $ u,v \in S$. Therefore $ab\mathrel\mathcal{J}^{+}ba$.

Also $a+ab=a+a'+a+ab $ and $a=a+a'+a=a+a'+(a+a')(b+b')+(a+a')+a=a+a'+ab+ab'+a'b+a'b'+a=a'+(a+ab)+ab'+a'b+a'b'+a$. Therefore a$\mathrel\mathcal{J}^+a+ab$ and hence $\mathcal{J}^{+}$ is a distributive lattice congruence on $S$. Now additive reduct of each $\mathcal{J}^+$ class is an orthodox completely regular semigroup. Therefore additive reduct of each $\mathcal{J}^+$ class is a rectangular group. Also each $\mathcal{J}^+$ class is a completely regular semiring. So  each $\mathcal{J}^{+}$-class is a rectangular ring. So $S$ is a distributive lattice of rectangular rings.

Conversely, suppose that $S$ be a distributive lattice $D$ of rectangular rings $\left\{S_{\alpha}\right\}_{\alpha \in D }$.  Let $e, f \in E^{+}(S)$. Suppose $e\in S_{\alpha}$ and $f\in S_{\beta}$ for some $\alpha,\beta \in D$. Then $e + ef, e+ fe \in S_{\alpha + \alpha \beta} = S_\alpha$. Since additive reduct $(S_{\alpha}, +)$ is a rectangular idempotent semiring, $e + ef + e = e $ and $e + fe + e = e$. So $S$ is a band orthoring.
\end{proof}

A semiring $S$ is said to be a left ring if it is a direct product of a left zero idempotent semiring and a ring.

\begin{theorem}
A completely regular semiring $S$ is a left band orthoring if and only if it is a distributive lattice of left rings.
\end{theorem}

A semigroup is a semilattice of groups if and only if it is a strong semilattice of groups. Such a semigroup is called Clifford semigroup. But the situation is not the same in case of semirings which are distributive lattices of rings. For every $ a\in S$, denote $V^+(a)=\{x \in S \mid a=a+x+a \; \textrm{and} \; x=x+a+x\}$. A semiring $S$ is a strong distributive lattice of rings if $(S, +)$ is abelian and $S$ satisfies the following conditions: for every $a, b \in S$, there exists $a' \in V^+(a)$ and $b' \in V^+(b)$ such that
\begin{eqnarray}
a          &= a+a'+a                           \label{4.1}        \\
a(a+a')    &= a+a'                             \label{4.2}        \\
a(b+b')    &= (b+b')a                          \label{4.3}        \\
a+a(b+b')  &= a                                \label{4.4}        \\
E^+(S) \; & \textrm{is a $k$-ideal of $S$}       \label{4.5}
\end{eqnarray}

A semiring which is a strong distributive lattice of rings is known as \emph{Clifford semiring} \cite{gho}, \cite{SMSclifford}. It is easy to check that a semiring $S$ is a distributive lattice of rings if $(S, +)$ is abelian and it satisfies the equations (\ref{4.1}) -- (\ref{4.4}). We call a semiring \emph{weak Clifford} if it is a distributive lattice of rings. Every weak Clifford semiring $S$ is an orthoring and $E^+(S)$ is a distributive lattice.

We left the checking that the equivalence relation $\nu^+$ defined on an orthoring $S$ by: for every $a, b \in S$,
$$ a \nu^+ b \; \textrm{if} \; a=a^o+b+a^o \; \textrm{and} \; b=b^o+a+b^o  $$
is the least $b$-lattice of rings congruence on $S$. In fact, this is the least weak Clifford semiring congruence on $S$, which we show in the following result.

\begin{theorem}  \label{band orthoring}
Every band orthoring is a spined product of a band semiring and a weak Clifford semiring with respect to a distributive lattice and conversely.
\end{theorem}
\begin{proof}
Suppose $S$ be a band orthoring. Then $\mathcal{H}^+$ is a band semiring congruence. Again $\nu^+$ is a b-lattice of rings congruence in an orthoring. Also in $S$, we have $(ab^o)^o+b^oa+(ab^o)^o=a^ob^o+b^oa^o+a^ob^o=a^ob^o=ab^o$ and $(b^oa)^o+ab^o+(b^oa)^o=b^oa$. Thus $ab^o\mathrel\nu^+ b^oa$. In addition we have $a^o+a+ab^o+a^o=a+a^o+a^ob^o+a^o=a+a^o=a$ and $(a+ab^o)^o+a+(a+ab^o)^o=a^o+a^ob^o+a+a^o+a^ob^o=a^o+a^ob^o+a^o+a+a^ob^o=a^o+a+a^ob^o=a+ab^o$ in $S$. This implies that $(a+ab^o)\nu^+ a$. Therefore $\nu^+$ is a weak Clifford semiring congruence on $S$. Denote $S_1=S/\mathcal{H}^+$ and $S_2=S/\nu^+$. Also $T=S/\mathcal{D}^+$ is a distributive lattice. Since $\mathcal{H}^+$, $\nu^+ \subseteq \mathcal{D}^+$, it follows that $\phi_1:S_1\longrightarrow T$ and $\phi_2:S_2\longrightarrow T$ defined by $\phi_1(\tilde{a})=\bar{a}$ and $\phi_2(\tilde{\tilde {a}})=\bar{a}$ are well defined surjective homomorphisms, where $\tilde {a} \,[ \tilde{\tilde {a}}, \bar{a} ]$ is the $\mathcal{H}^+$ [ $\nu^+$, $\mathcal{D}^+ ]$ -class containing $a$. Thus $R = \{(\tilde{a}, \tilde{\tilde{b}}) \in S_1 \times S_2 \mid a \mathrel\mathcal{D}^+ b \}$ is a spined product of $S_1$ and $S_2$ with respect to $T$.

Then the mapping   $\theta:S\longrightarrow R$ defined by $\theta(a)=(\tilde{a}, \tilde{\tilde{a}})$ is a monomorphism. Also if $( \tilde{a}, \tilde{\tilde{b}} ) \in R$, then we have $a \mathrel\mathcal{D}^+ b$. Since $\mathcal{D}^+=\mathcal{H}^+o \nu^+$, there exists $c\in S$ such that $a\mathrel\mathcal{H}^+ c$ and $c \mathrel\nu^+ b$. This implies that $\theta(c)=( \tilde{a}, \tilde{\tilde{b}} )$. Thus $\theta$ is onto.
\end{proof}

Similarly we have the following generalization.
\begin{theorem}
A semiring $S$ is a left band orthoring if and only if $S$ is a spined product of a left band semiring and a weak Clifford semiring with respect to a distributive lattice.
\end{theorem}

On a band orthoring, we define two equivalence relations $\lambda_1$ and $\lambda_2$ by: for all $a$, $b\in S$,
\begin{align*}
                   & a\mathrel\lambda_1b \hspace{0.2cm}\mbox{if}\hspace{0.2cm} a=a^o+b\hspace{0.2cm} \mbox{and} \hspace{0.2cm} b=b^o+a;\\
\textrm{and} \; \; & a\mathrel\lambda_2b \hspace{0.2cm}\mbox{if} \hspace{0.2cm}a=b+a^o\hspace{0.2cm} \mbox{and} \hspace{0.2cm} b=a+b^o.
\end{align*}
Then we have the following theorem.
\begin{theorem}                               \label{band orthoring lrc}
Every band orthoring is a spined product of a left band orthoring and a right band orthoring with respect to a weak Clifford semiring.
\end{theorem}
\begin{proof}
Let $S$ be a band orthoring. Since additive reduct of a band semiring is a regular band \cite{Wang}, $S$ is a regular orthoring. So $S$ is a spined product of $S/\lambda_{1}$ and $S/\lambda_{2}$. Suppose $S_{1} = S/\lambda_{2}$ and $ S_{2} = S/\lambda_{1}$ \cite{PR vol I}. Let $e\lambda_{2}, f\lambda_{2}\in E^{+}(S_{1})$. Then $e + fef + e^{0} = e + fef + e = e$ and $e + (e + fef)^{0} = e + fef$. Therefore $e \lambda_{2} (e+fef)$ and hence $E^{+}(S_{1}) \in \mathbf{LBI}$. Similarly $E^{+}(S_{2})\in \mathbf{RBI}$. Also, we have $\lambda_{1} o \lambda_2=\lambda_2 o \lambda_1= \nu ^+$. Since $\nu ^+$ is a weak Clifford semiring congruence on $S$, $S$ is spined product of a left band orthoring  and a right band orthoring with respect to a weak Clifford semiring.
\end{proof}

\begin{theorem}
An orthoring $S$ is a subdirect product of a band semiring and a ring if and only if $S$ satisfies
\begin{itemize}
\item[(i)]
$a(a+a')=a+a'$;
\item[(ii)]
$E^+(S)$ is a k-ideal of $S$;
\item[(iii)]
$(a+a')+b(a+a')b+(a+a')=(a+a')$.
\end{itemize}
\end{theorem}
\begin{proof}
An orthoring $S$ is a subdirect product of an idempotent semiring $T$ and a ring $R$ if and only if $S$ satisfies (i) and (ii). We now show that $T\in \mathbf{BI}$ if and only if $S$ satisfies (iii). First suppose that $T\in \mathbf{BI}$. Let $a,b\in S$. Then $a=(e,x), b=(f,y)$  for some $e,f\in T$ and $x,y\in R$.

Now, $(a+a')+b(a+a')b+(a+a')=((e,x)+(e,-x))+(f,y)((e,x)+(e,-x))(f,y)+((e,x)+(e,-x))
=(e,0_R)+(f,y)(e,0_R)(f,y)+(e,0_R)=(e+fef+e,0_R)=(e,0_R)=(e,x)+(e,-x)=a+a'$.

Conversely, suppose that $S$ satisfies (iii). Let $e,f\in T$. Then there exists $x,y\in R$ such that $(e,x),(f,y)\in S$. So,
\begin{align*}
             &((e,x)+(e,-x))+(f,y)((e,x)+(e,-x))(f,y)+((e,x)+(e,-x))=(e,x)+(e,-x)\\
\Rightarrow  &(e+fef+e,0_R)=(e,0_R)\\
\Rightarrow  &e+fef+e=e.
\end{align*}
\end{proof}

\begin{corollary}
An orthoring $S$ is a subdirect product of a left quasi band semiring and a ring if and only if $S$ satisfies
\begin{itemize}
\item[(i)]
$a(a+a')=a+a'$;
\item[(ii)]
$(a+a')+(a+a')(b+b')+(a+a')=(a+a')$;
\item[(iii)]
$a\in S, a+b=b $ for some $b\in S$ implies $a+a=a$.
\end{itemize}
\end{corollary}

\begin{corollary}
An orthoring $S$ is a subdirect product of a left band semiring and a ring if and only if $S$ satisfies
\begin{itemize}
\item[(i)]
$a(a+a')=a+a'$;
\item[(ii)]
$b(a+a')b+(a+a')=(a+a')$;
\item[(iii)]
$a\in S, a+b=b $ for some $b\in S$ implies $a+a=a$.
\end{itemize}
\end{corollary}

\section{Variety of band orthorings}

Let us first fix the following symbols:
\begin{center}
\begin{tabular}{ll}
   $\mathbf{Ri}$&\hspace{0.05 cm}   Variety of rings;\\
  $\mathbf{BOR}$&\hspace{0.05 cm}   Variety of band orthorings;\\
  $\mathbf{LBOR}$& \hspace{0.05 cm}  Variety of left band orthorings;\\
  $\mathbf{RBOR}$ &\hspace{0.05 cm}  Variety of right band orthorings.
\end{tabular}
\end{center}

Then the following result follows immediately from Theorem \ref{band orthoring lrc}.
\begin{theorem}
$\mathbf{BOR}=\mathbf{LBOR}\vee \mathbf{RBOR}$.
\end{theorem}

Now we characterize the interval $[\,\mathbf{Ri}, \mathbf{BOR}]$. First we have the following result which follows from the Corollary 3.2 \cite{pastijn-guo}.
\begin{lemma} \label{subvariety}
Let $\mathcal{U}$ be any subvariety of $\mathbf{BI}$. Then $\mathcal{U}\vee \mathbf{Ri}$ is the subvariety of all the orthorings $S$ for which $E^+(S)\in \mathcal{U}$.
\end{lemma}

The following consequence is useful.
\begin{lemma}  \label{SR }
For any subvariety $\mathcal{U}$ of $\mathbf{BI}$ and any subvariety $\mathcal{V}$ of $[ \mathbf{Ri}, \mathbf{BOR} ]$, $\mathcal{U}= (\mathcal{U}\vee \mathbf{Ri})\cap\mathbf{BI}$ and $\mathcal{V}= (\mathcal{V}\cap \mathbf{BI})\vee \mathbf{Ri}$.
\end{lemma}
\begin{proof}
Clearly $\mathcal{U}\subseteq (\mathcal{U} \vee\mathbf{Ri})\cap\mathbf{BI}$. On the other hand if $S\in (\mathcal{U}\vee \mathbf{Ri})\cap\mathbf{BI}$, then $S\in \mathbf{BI}$ and $S$ is a homomorphic image of a subdirect product of a semiring in $\mathcal{U}$ and a ring. Then by Lemma \ref{subvariety}, there is an orthoring $T$ with $E^+(T)\in \mathcal{U}$ and a homomorphism $\phi: T\rightarrow S$. Again by Lallements lemma, $\phi(E^+(T))=E^+(S)$ and hence $S=E^+(S)\in \mathcal{U}$. Thus $(\mathcal{U}\vee \mathbf{Ri})\cap\mathbf{BI} \subseteq \mathcal{U}$.

Since $\mathcal{V}\cap\mathbf{BI}, \mathbf{Ri} \subseteq \mathcal{V}$, we have $(\mathcal{V}\cap \mathbf{BI})\vee \mathbf{Ri}\subseteq \mathcal{V}$. Now suppose that $S\in \mathcal{V} \subseteq [\mathbf{Ri}, \mathbf{BOR}]$. Then $S$ is a band orthoring. So by Theorem 3.3 \cite{pastijn-guo}, $S$ is a homomorphic image of a subdirect product of a ring $R$ and $E^+(S)$. Again since $E^+(S)$ is a subsemiring of $S$ and $S\in \mathcal{V}$, $E^+(S)\in \mathcal{V}$. Therefore $S\in (\mathcal{V}\cap \mathbf{BI})\vee \mathbf{Ri}$.
\end{proof}

Define two mappings $\phi: [\mathbf{Ri}, \mathbf{BOR}]\rightarrow \mathcal{L}(\mathbf{BI})$ such that $\phi(\mathcal{V})=\mathcal{V}\cap\mathbf{BI}$ and $\psi: \mathcal{L}(\mathbf{BI})\rightarrow [\mathbf{Ri}, \mathbf{BOR}]$ such that $\psi(\mathcal{U})=\mathcal{U}\vee \mathbf{Ri}$. Then by Lemma \ref{SR }, $\phi$ and $\psi$ are both one-one and onto mappings such that $\phi o \psi=\iota=\psi o \phi$. Thus we have the following theorem:

\begin{theorem}
The interval $[\,\mathbf{Ri}, \mathbf{BOR}]$ is isomorphic to $\mathcal{L}(\mathbf{BI})$.
\end{theorem}

The following result can be proved similarly.

\begin{theorem}
The interval $[\,\mathbf{Ri}, \mathbf{LBOR}]$ is isomorphic to $\mathcal{L}(\mathbf{LBI})$ and the interval $[\,\mathbf{Ri}, \mathbf{RBOR}]$ is isomorphic to $\mathcal{L}(\mathbf{RBI})$.
\end{theorem}

\bibliographystyle{amsplain}

\end{document}